\documentclass[12pt,a4paper,oneside]{amsart}

\usepackage{anysize}
\marginsize{2cm}{2cm}{1.5cm}{1.5cm}
\newtheorem{thm}{Theorem}[section]
\newtheorem{prp}[thm]{Proposition}
\newtheorem{lem}[thm]{Lemma}
\newtheorem{cor}[thm]{Corollary}
\newtheorem{conj}[thm]{Conjecture}

\theoremstyle{definition}
\newtheorem{dfn}{Definition}[section]
\newtheorem{rem}[dfn]{Remark}
\newtheorem{exl}[dfn]{Example}

\newcommand{\st}{:\;}

\def\R{{\mathbb R}}%

\newcommand{\Red}{\R^d}
\renewcommand{\phi}{\varphi}

\newcommand{\ball}[1]{\mathbf{B}^{#1}}

\providecommand{\parenth}[1]{\left(#1\right)}%
\providecommand{\braces}[1]{\left\{#1\right\}}%
\newcommand{\iprod}[2]{\left\langle#1,#2\right\rangle}%
\def\polar{\circ}
\newcommand{\polarset}[1]{{#1}^{\polar}}%

\newcommand{\conv}{\mathrm{conv}}%

\providecommand{\parenth}[1]{\left(#1\right)}%
\providecommand{\braces}[1]{\left\{#1\right\}}%
\newcommand{\vol}[1]{\operatorname{vol}\nolimits_{#1}}%

\newcommand{\truncf}[1]{
\left(#1 \right)_{\! \! +}
}%

\usepackage{xcolor}
\usepackage{hyperref}
\hypersetup{
	colorlinks   = true, %Colours links instead of ugly boxes
	urlcolor     = blue, %Colour for external hyperlinks
	linkcolor    = blue, %Colour of internal links
	citecolor    = red %Colour of citations
}
\newcommand{\Href}[2]{\hyperref[#2]{#1~\ref{#2}}}

\title{Quantitative Steinitz theorem and polarity}
\author{Grigory Ivanov\address{Grigory Ivanov: 
Pontifícia Universidade Católica do Rio de Janeiro \\
Departamento de Matemática,
Rua Marquês de São Vicente, 225\\
Edif{\'i}cio Cardeal Leme, sala 862,
22451-900 G{\'a}vea, Rio de Janeiro, Brazil}
\email{grimivanov@gmail.com}}
\thanks{The author is supported by Projeto Paz and Coordenacao de Aperfeicoamento de Pessoal de Nivel Superior - Brasil (CAPES) - 23038.015548/2016-06}
\begin{document}

\begin{abstract}

The classical Steinitz theorem asserts that if the origin lies within the interior of the convex hull of a set $S \subset \mathbb{R}^d$, then there are at most $2d$ points in $S$ whose convex hull contains the origin within its interior. B\'ar\'any, Katchalski, and Pach established a quantitative version of Steinitz's theorem, showing that for a convex polytope $Q$ in $\mathbb{R}^d$ containing the standard Euclidean unit ball $\mathbf{B}^d$, there exist at most $2d$ vertices of $Q$ whose convex hull $Q'$ satisfies $r\mathbf{B}^d \subset Q' $ with $r \geq d^{-2d}$. Recently, M\'arton Nasz\'odi and the author derived a polynomial bound on $r$.

This paper aims to establish a bound on $r$ based on the number of vertices of $Q.$
 In other words, we demonstrate an effective method to remove several points from the original set $Q$ without significantly altering the bound on $r$. Specifically, if the number of vertices of $Q$ scales linearly with the dimension, i.e., $\alpha d$, then one can select $2d$ vertices such that $r \geq \frac{1}{5 \alpha d}$. The proof relies on a polarity trick, which may be of independent interest: we demonstrate the existence of a point $c$ in the interior of a convex polytope $P \subset \mathbb{R}^d$ such that the vertices of the polar polytope $(P-c)^\circ$ sum up to zero.
\end{abstract}

\maketitle

\section{Introduction}
The goal of this paper is to establish a quantitative version of the following classical result of  
E. Steinitz \cite{steinitz1913bedingt}.
\begin{prp}[Steinitz theorem]
Let the origin belong to the interior of the convex hull of a set $S \subset \Red.$ 
Then there are at most $2d$ points of $S$ whose convex hull contains the origin in the interior. 
\end{prp}

The first quantitative version of this result was obtained in \cite{barany1982quantitative}, where the authors showed that for a convex polytope $Q$ in $\mathbb{R}^d$ containing the standard Euclidean unit ball $\mathbf{B}^d$, there exist at most $2d$ vertices of $Q$ whose convex hull $Q'$ satisfies $r(d)\mathbf{B}^d \subset Q' $ with $r(d) \geq d^{-2d}$.

With the exception of the planar case $d = 2$ \cite{kirkpatrick1992quantitative, brass1997quantitative, barany1994exact}, no significant improvement on $r(d)$ had been obtained until recently (see also \cite{de2017quantitative}).
M\'arton Nasz\'odi and the author derived the first polynomial lower bound
$r(d) \geq \frac{1}{6d^2}$ in \cite{ivanov2024quantitative}, and extended this result to a spherical version in \cite{ivanov2023quantitative}.
The current working conjecture is that $r(d) \geq \frac{\alpha}{\sqrt{d}}$ for some positive constant $\alpha$.

The main result of the paper is as follows.
\begin{thm}\label{thm:Steinitz_dropping_out_points}
Let $Q $ be a set of $m$ points of $\R^d$ such that its convex hull $\conv{\ \!Q}$  contains the Euclidean unit ball $\ball{d}.$   Then there is $Q^\prime \subset Q$ of size at most $2d$ satisfying
$\conv{\ \! Q^\prime} \supset r \ball{d},$ where $r = \frac{1}{2(m+d)+1}.$
\end{thm}

Starting with the breakthrough \cite{BSS14}, which led to new results in the area of quantitative combinatorial convexity (see \cite{de2017quantitative}, \cite{naszodi2016proof}, \cite{Brazitikos2016Diam}, \cite{brazitikos2018polynomial}, \cite{brazitikos2017brascamp}, \cite{fernandez2022continuous}), one approach to the problems was to initially identify more than $2d$ points (facets, subsets) with desired properties, typically linear in the dimension, and then select the best $2d$ among them. It is worth noting that in some cases, eliminating additional objects poses challenges \cite{damasdi2021colorful}. The next corollary, which trivially follows from the main result,  facilitates this process in the case of the Quantitative Steinitz  theorem.

\begin{cor}
Let $Q$ be a set of $\alpha d,$ $\alpha > 1,$ points of $\R^d$ such that  its convex hull $\conv{\ \!Q}$  contains the ball $\lambda \ball{d}.$ 
Then there are at most $2d$ points of $Q$ whose convex hull
$Q^\prime$ satisfies
\[
\frac{\lambda}{5 \alpha d} \ball{d} \subset Q^\prime.
\]  
\end{cor}

As another elementary corollary of the main result, we will get a slightly worse polynomial bound in the Quantitative Steinitz theorem than the quadratic one obtained in \cite{ivanov2024quantitative}.
\begin{cor}\label{cor:QST_five_halfs}
Let $Q$ be a convex polytope  in $\Red$ containing the Euclidean unit ball $\ball{d}.$
Then there are at most $2d$ vertices of $Q$ whose convex hull
$Q^\prime$ satisfies
\[
\frac{d^{-\frac{5}{2}} }{7} \ball{d} \subset Q^\prime.
\]  
\end{cor}

The key observation we will use to prove \Href{Theorem}{thm:Steinitz_dropping_out_points} is the following ``polarity trick.'' 

We recall that the \emph{polar} of a set $S \subset \Red$ is defined by
\[
\polarset{S} = \braces{x \in \Red \st \iprod{x}{s} \leq 1 \quad \text{for all} \quad s \in S}.
\]
\begin{thm}\label{thm:polarity_trick_sum_of_vertices}
Let  $P \subset \Red$ be a polytope with  non-empty interior.
Then there is a point $c$ in its interior such that 
the sum of vertices of $\polarset{\parenth{P-c}}$ is equal to zero.
\end{thm}

In fact, we will show that for any positive weights, there is a point $c$ from the interior of $P$ such that the sum of vertices of $\polarset{\parenth{P-c}}$ with those weights is zero.
We will show that the corresponding point $c$ is a maximizer of a certain functional. Thus, our proof mimics the proof of the existence of the \emph{Santal{\'o} point} (see \cite{GruberBook, meyer1998santalo, lehec2009partitions, ivanov2021geometric}), which is a point $s$ inside a convex set $K \subset \R^d$ with non-empty interior such that the centroid of $\parenth{K-s}^\polar$ is the origin.

The author hopes that the bound on $r$ in Theorem 1.2 is not optimal, since it would contradict the conjectured lower bound for the quantitative version of the Steinitz theorem.

An interesting example, showing that even in the case of a set of $2d + 1$ points the conjectured bound on $r(d)$ is attained, was communicated to the author by Florian Grundbacher:
\begin{exl}
Define $p_{-} = - \sqrt{d} (e_d + e_1 + \dots + e_{d-1})$ and
$p_{+} = - \sqrt{d} (e_d - e_1 - \dots - e_{d-1}),$ wheere $e_1, \dots, e_d$ is the standard basis of $\R^d.$ 
Set $Q = \left\{\pm\sqrt{d} e_1, \dots, \pm \sqrt{d} e_{d-1}, \sqrt{d} e_d,
p_{-}, p_{+}  \right\}.$ Then $Q$ consists of $2d+1$ points and its convex hull contains the unit ball $\ball{d}.$ Moreover, any $Q^\prime \subset Q$ of size at most $2d$ does not contain the ball $r \ball{d}$ for any 
$r > \sqrt{\frac{d}{d^2 + d-1}}.$
\end{exl}

The rest of the paper is organized as follows:
In the next Section, we will explain the ideas behind the proof of the Quantitative Steinitz theorem obtained in \cite{ivanov2024quantitative} that can be traced back to  \cite{ivanov2022quantitative} and \cite{almendra2022quantitative}; we will try to show why \Href{Theorem}{thm:polarity_trick_sum_of_vertices} comes naturally as a development of those ideas. 
In \Href{Section}{sec:polarity_trick}, we will prove a more general version of \Href{Theorem}{thm:polarity_trick_sum_of_vertices}. Finally, in \Href{Section}{sec:proof_of_the_main_result} we  derive \Href{Theorem}{thm:Steinitz_dropping_out_points} and  its corollary.

\section{Useful lemmas}
\label{sec:usefull_lemmas}

We begin by elucidating the proof ideas of the Quantitative Steinitz theorem as outlined in \cite{ivanov2024quantitative}. The central strategy revolves around the careful application of polarity, employed twice in succession.

We started with a ``Steinitz-type picture'', wherein we considered a set $Q \subset \mathbb{R}^d$ whose convex hull contains the unit ball $\ball{d}$. Subsequently, we transitioned to an equivalent  ``Helly-type picture'' by examining the polar set $\polarset{Q}$ of $Q.$ This transformation allows us to reformulate the original problem into an equivalent Helly-type statement, justifying the name.
Now comes a trick:  we chose a point $c$  ``deep'' in $\polarset{Q}$ and considered 
$\polarset{\parenth{\polarset{Q} - c}}.$ So to say, this maneuver returns us to a ``Steinitz-type picture''  albeit a modified one, as we have altered our original set.
In essence, by manipulating the center of polarity, we achieve a more structurally organized convex polytope. We dub the resultant configuration following the second polarity transformation as ``Atlantis.'' 
For a new set within ``Atlantis'', we derived the desired polynomial bound utilizing a result from  \cite{almendra2022quantitative}. 
Finally, we demonstrated that reverting to the original  ``Steinitz-type picture'' 
does not significantly degrade our bound.

The crux of the proof lies in selecting the appropriate center $c$ of polarity during the second step. Notably, \Href{Theorem}{thm:polarity_trick_sum_of_vertices} offers a methodology for choosing an alternative point, which holds intrinsic interest in itself. 

Now, we are going to formalize a few statements.

For a positive integer $n,$ $[n]$ denotes the set $\{1, \dots, n\};$ $\ball{d}$ denotes the standard Euclidean unit ball in $\R^d;$ $\iprod{p}{x}$ denotes the inner product of $p$ and $x.$ We use $\truncf{a}$ to denote $\max\{a, 0\}.$

We start with an open problem.
In relation to volumetric Helly-type results, the author is interested in the following conjecture:
Macbeath \cite[Lemma 7.1]{macbeath1952theorem} showed that for a compact convex set $K \subset \R^d$ with non-empty interior, the function 
$f(x) = \vol{d} \left(K \cap (-K + 2x)\right)$ attains its maximum in a unique point of the interior of $K$ (here $\vol{d}$ denotes the $d$-dimensional volume on $\R^d$, as usual). Let us call this point the \emph{Macbeath point} of $K.$
\begin{conj}
The Macbeath point $p$ of a  compact convex set $K \subset \R^d$ with non-empty interior satisfies the inclusion
\[
K - p \subset -d (K-p). 
\]
\end{conj} 
\begin{rem}
To the best of our knowledge, all known quantitative Carathéodory-type and Helly-type results require the consideration of a certain center ``deep'' inside a convex body. In the current manuscript, we use the point from Theorem 1.5. Other notable examples include the center of the John ellipsoid, the barycenter, the Santaló point, etc.  
The author believes that the Macbeath point might open a new direction for both volumetric Helly-type results and ``coarse'' approximations of convex bodies.  
\end{rem}

We formulate now the above-mentioned result from \cite{almendra2022quantitative}.

\begin{lem}\label{lem:max_volume-o-sipmex_inclustions}
Let $L$ be a bounded subset of $\R^d$ linearly spanning the whole space,
let $S = \conv\{0,v_1,\ldots,v_d\}$ be the maximal volume simplex among all simplices with $d$ vertices from  $L$ and one vertex at the origin. We use $P$ to denote the Minkowski sum of segments $[-v_i, v_i],$ $i \in [d].$ Then the following inclusions hold:
$$
        L \subset P \subset -2dS + (v_{1}+ \dots + v_{d}).
$$

\end{lem}
%\begin{proof}[Sketch of the proof]
%Clearly, the volume of $S$ is strictly positive. 
%The simplex $S$ can be represented as
%    \begin{equation}\label{eq:conv_S}
%        S = \braces{ x \in \Red \st x = \alpha_1 v_1 + \ldots + \alpha_d v_d \quad  \textrm{ for }\  \alpha_i \geq 0 \textrm{ and } \sum_{i=1}^d \alpha_i \leq 1 }.
%    \end{equation}
%         It is easy to see that
%     $P$ is a paralletope that can be represented as
%    \begin{equation}\label{eq:conv_P}
%        P = \{ x \in \Red \st x = \beta_1 v_1 + \ldots + \beta_d v_d \quad \textrm{ for } \beta_i \in [-1,1] \}.
%    \end{equation}
%Since $S$ is chosen maximally, equation \eqref{eq:conv_P} shows that for any vertex $v$ of $L$, $v \in P$. By convexity,
%    \begin{equation*}
%       % \label{eq:Q_subset_P_ambrus}
%        L \subset P.
%    \end{equation*}
%    Let $S^\prime = -2dS + (v_1 + \ldots + v_d)$. By \eqref{eq:conv_S},
%    \begin{equation*}%\label{eq:conv_S'}
%        S^\prime = \braces{x \in \Red \st x = \gamma_1v_1 + \ldots + \gamma_dv_d \quad  \text{ for } \gamma_i \leq 1 \textrm{ and } \sum_{i= 1}^{d} \gamma_i \geq -d },
%    \end{equation*}
%    which, together with \eqref{eq:conv_P}, yields
%$
%        P \subseteq S^\prime,
%$ completing the proof.
%\end{proof}

Now we want to show that the whole way from ``Steinitz-type picture'' to ``Atlantis'' and back does not cost much in terms of the bound on the radius.  

Let $P$ be a polytope in $\R^d$ with a non-empty interior. 
It is well known that for any point $c$ of the interior of $P,$ there is a one-to-one correspondence between the facets of $P$ and the vertices $\polarset{\parenth{P - c}}.$
For two points $c_1$ and $c_2$ of the interior of $P,$ we will say that a vertex of $\polarset{\parenth{P - c_1}}$ and a vertex of $\polarset{\parenth{P - c_2}}$
are \emph{polar corresponding} if they correspond to the same facet of $P.$

\begin{lem}[Vertex correspondence]\label{lem:polar_vertex_correspondence}
Let  $P \subset \Red$ be a polytope containing the origin and a point $c$ in its interior. Denote $Q = \polarset{P}$ and 
$L = \polarset{(P-c)}.$ Then $v$ is a vertex of $Q$ if and only if 
$\frac{v}{ 1 - \iprod{c}{v}}$ is a vertex of $L.$ Moreover, the vertex $v$ of $Q$ and the vertex $\frac{v}{ 1 - \iprod{c}{v}}$ of $L$ are polar corresponding.
\end{lem}
\begin{proof}
A point $v$ is a vertex of $Q$ if and only if the half-space 
$H_v = \{x \in \Red \st \iprod{x}{v} \leq 1\}$ supports $P$ by  a facet.
The latter is true if and only if $H_v -c$ supports  $P -c$ in a  facet. On the other hand, 
since $c$ is in the interior of $P,$ $\iprod{c}{v} < 1,$ and thus,
\[
H_v - c = \{x \in \Red \st \iprod{x}{v} \leq 1\} - c = 
\{y \in \Red \st \iprod{y}{v} \leq 1 - \iprod{c}{v}\} =
\]
\[
\{y \in \Red \st \iprod{y}{\frac{v}{ 1 - \iprod{c}{v}}} \leq 1\}.
\]
Consequently, $\frac{v}{ 1 - \iprod{c}{v}}$ is a vertex of $L$ if and only if 
$v$ is a vertex of $Q.$
\end{proof}

\begin{lem}[Atlantis: There and back again]\label{lem:from_Atlantida_to_Steinitz}
Let  $P \subset \ball{d} \subset \Red$ be a polytope containing the origin and a point $c$ in its interior. Denote $K_1 = \polarset{P}$ and 
$K_2 = \polarset{(P-c)}.$ If some vertices $w_1, \dots, w_k$ of $K_2$ satisfy 
the inclusion $\conv{\{w_1, \dots, w_k\}} \supset  \lambda \ball{d}$ for some positive $\lambda,$ then
their polar corresponding vertices $v_1, \dots, v_k$ of $K_1$ satisfy
$\conv{\{v_1, \dots, v_k\}} \supset \frac{\lambda}{1+\lambda}\ball{d}. $
\end{lem}
\begin{proof}
By \Href{Lemma}{lem:polar_vertex_correspondence}, the   vertices $v_1, \dots, v_k$ of $K_1$ polar corresponding to $w_1, \dots, w_k$ satisfy
$w_1 = \frac{v_1}{ 1 - \iprod{v_1}{c}}, \dots, w_k = \frac{v_k}{ 1 - \iprod{v_k}{c}}.$
Next, $\polarset{\parenth{\conv{\{w_1, \dots, w_k\}}}} \subset \frac{1}{\lambda}\ball{d}$ and
\[\polarset{\parenth{\conv{\{w_1, \dots, w_k\}}}} =  \bigcap\limits_{i \in [k]}
\braces{y \in \Red \st \iprod{y}{w_i} \leq 1} = 
\bigcap\limits_{i \in [k]}
\braces{y \in \Red \st \frac{\iprod{y}{v_i}}{ 1 - \iprod{v_i}{c}} \leq 1} =
\] 
\[
\bigcap\limits_{i \in [k]}
\braces{y \in \Red \st \iprod{y}{v_i} \leq 1 - \iprod{v_i}{c}}
=
\bigcap\limits_{i \in [k]}
\braces{y \in \Red \st \iprod{y +c}{v_i} \leq 1 }
=
\]
\[
\bigcap\limits_{i \in [k]}
\parenth{\braces{x \in \Red \st \iprod{x}{v_i} \leq 1 } -c}  = 
 -c + \bigcap\limits_{i \in [k]}
\braces{x \in \Red \st \iprod{x}{v_i} \leq 1 }.
\]
By the assumption of the lemma, $c \in P \subset \ball{d}.$ Hence,
$\bigcap\limits_{i \in [k]}
\braces{x \in \Red \st \iprod{x}{v_i} \leq 1 } \subset \parenth{\frac{1}{\lambda} +1} \ball{d}.$ Consequently, 
$\conv{\{v_1, \dots, v_k \}} = \polarset{\parenth{\bigcap\limits_{i \in [k]}
\braces{x \in \Red \st \iprod{x}{v_i} \leq 1 }}} \supset  \frac{\lambda}{1+\lambda}\ball{d}.$ The lemma is proven.
\end{proof}
Since the roles of the origin and $c$
above are symmetric, \Href{Lemma}{lem:from_Atlantida_to_Steinitz} allows to go both ways from ``Steinitz-type'' picture to ``Atlantis'' and back. 
For example, if $K_2 \supset \ball{d},$ then $K_1 \supset \frac{1}{2} \ball{d}.$

\section{Polarity trick}\label{sec:polarity_trick}
In this section, we prove the following result, which implies \Href{Theorem}{thm:polarity_trick_sum_of_vertices}.
\begin{thm}
\label{thm:polarity_trick_weighted_sum_of_vertices}
Let  $P \subset \Red$ be a polytope  with facets $F_1, \dots, F_n$ and non-empty interior.
For any positive weights $\alpha_1, \dots, \alpha_n,$ there is a point $c$ in the interior of $P$ such that 
\[
\sum \limits_{i \in [n]} \alpha_i w_i =0,
\]
where $w_i$ is  the vertex of $\polarset{\parenth{P-c}}$ corresponding to the facet $F_i$ of $P.$
\end{thm}
The key observation is the following:
\begin{lem}\label{lem:sum_of_vertexset_weighted}
Let  $P \subset \Red$ be a polytope containing the origin in its interior defined by the linear inequalites $\{x \in \Red \st \iprod{x}{v} \leq 1 \text{ for all } v\in Q\}$, for some finite $Q\subset\Red$.
For a given set of positive numbers 
$\beta_v > 0,$ $v \in  Q,$ we
introduce the function 
\[
F(x) = \prod\limits_{v \in Q} \truncf{1 -  \iprod{x}{v}}^{\beta_v}.
\]
 Then $F$ attains its maximum at a unique point $c$ of the interior of $P$ satisfying the identity
\[
\sum\limits_{v \in Q}\frac{\beta_v v}{1 - \iprod{c}{v}} =0.
\]
\end{lem}
\begin{proof}
Clearly, the function $F$ vanishes outside the interior of $P$.  
The function $F$ is smooth, and the identity
\[
F(x) = \prod\limits_{v \in Q} \parenth{1 -  \iprod{x}{v}}^{\beta_v}
\]
holds for any $x$ in the interior of $P$.  
By compactness, $F$ attains its maximum at a point $c$ in the interior of $P$.  
The function $\ln F$ is strictly concave on its support, which implies the uniqueness of $c$.  
By direct calculation, the gradient of $F$ is given by
\[
\nabla F(x) = F(x) \cdot \sum_{v \in Q} 
\frac{-\beta_v v}{1 - \iprod{x}{v}},
\]
which must vanish at the maximum point. This completes the proof of the lemma.
\end{proof}

\begin{proof}[Proof of \Href{Theorem}{thm:polarity_trick_weighted_sum_of_vertices} ]
Returning to our theorem, we shift $P$ in such a way that it contains the origin in its interior. Denote $Q  = \polarset{P}$ and let $v_i$ be the vertex of $Q$ corresponding to the facet $F_i.$
Applying \Href{Lemma}{lem:sum_of_vertexset_weighted} for 
$F(x) = \prod\limits_{i \in [m]} \truncf{1 - \iprod{x}{v_i}}^{\alpha_i},$ we get a point $c$ in the interior of $P$ that satisfies
\[
\sum\limits_{i \in [m]}\frac{\alpha_i v_i}{1 - \iprod{c}{v_i}} =0.
\]
By \Href{Lemma}{lem:polar_vertex_correspondence}, the sum is equal to 
$\sum\limits_{i \in [m]} {\alpha_i w_i},$ where $w_i$ is  the vertex of $\polarset{\parenth{P-c}}$ corresponding to the facet $F_i$ of $P.$ The proof of \Href{Theorem}{thm:polarity_trick_weighted_sum_of_vertices} is complete.
\end{proof}

\section{Proof of the main result and its corollary}
\label{sec:proof_of_the_main_result}

The following consequence of the Carath\'eodory lemma comes in handy. The proof can be found in \cite[Theorem 2.3]{barany2021combinatorial}.
\begin{lem}\label{lem:Caratheodory_minus_point}
Assume $b \in \R^d$ and a point $p$ belongs to the convex hull of a set $Q \subset \R^d.$
Then there are  $v_1, \dots, v_d$ (some of them might coincide) in $Q$ satisfying
$p \in \conv{\{b, v_1, \dots, v_d\}}.$ 
\end{lem}
\begin{proof}[Proof of \Href{Theorem}{thm:Steinitz_dropping_out_points}]
Our proof follows the same structure as the proof of the quantitative Steinitz theorem in \cite{ivanov2024quantitative}, with the main difference being that we choose a “center” deep inside the body using the point from Theorem~1.5.

Set $P = \polarset{Q}$. By \Href{Theorem}{thm:polarity_trick_sum_of_vertices}, there is a point $c$ 
in the interior of $P$ such that   the  vertices of  $\polarset{(P-c)}$  sum up to zero.   Denote $L =\polarset{\parenth{P-c}}.$  

Using \Href{Lemma}{lem:from_Atlantida_to_Steinitz} with  $K_2 = Q$ and $K_1 = L,$ one sees that $ \frac{\ball{d}}{2} \subset L.$

Consider  $S = \conv\{0,w_1,\ldots,w_d\}$  the maximal volume simplex among all simplices with $d$ vertices from  $L$ and one vertex at the origin. 
Then the sum of all other vertices  of $L$ is equal to
$- (w_1 + \dots + w_d).$ And thus the centroid $p$ of all others is equal to
$- \frac{(w_1 + \dots + w_d)}{m-d}.$ 
Thus, by \Href{Lemma}{lem:Caratheodory_minus_point}, there are vertices $w_{d+1}, \dots, w_{2d}$ such that $p$ belongs to $\conv\{w_{d+1}, \dots, w_{2d}, b\},$ where $b = \frac{w_1 + \dots + w_d}{d} \in L.$
By construction, the convex hull of $\{w_{1}, \dots, w_{2d}\}$ contains the origin since it belongs to the segment with endpoints $p$ and $b.$ 

By \Href{Lemma}{lem:max_volume-o-sipmex_inclustions}, 
\[
 \frac{\ball{d}}{2} \subset L \subset -2d   \conv{\{0, w_1, \dots, w_d\}} + (w_1 + \dots + w_d) \subset
-2d\conv{\{ w_1, \dots, w_{2d}\}} - p (m-d) \subset
\]
\[
 -2d  \conv{\{w_1, \dots, w_{2d}\}}  -
(m-d)  \conv{\{w_1, \dots, w_{2d}\}}  = - (m +d)   \conv{\{ w_1, \dots, w_{2d}\}}.
\]
Thus, $\frac{\ball{d}}{2(m+d)} \subset \conv{\{w_1, \dots, w_{2d}\}},$ 
and by \Href{Lemma}{lem:from_Atlantida_to_Steinitz}  with  $K_2 = L$ and $K_1 = Q,$ one sees that the corresponding vertices  $v_1, \dots, v_{2d}$ of $Q$ satisfy 
$\frac{\ball{d}}{2(m+d)+1} \subset \conv{\{ v_1, \dots, v_{2d}\}}.$ 
\end{proof}

\begin{proof}[Proof of  \Href{Corollary}{cor:QST_five_halfs}]
The first step is to reduce the number of points to a quadratic in $d.$
It is easy to find $2d^2$ points of $Q$ such that their convex hull contains 
$\frac{\ball{d}}{\sqrt{d}}.$ Take an arbitrary standard cross-polytope inscribed in the unit ball $\ball{d},$ say the convex hull of vectors of the standard basis
$\{e_1, \dots, e_d\}$ of $\R^d$ and their opposites $\{-e_1, \dots, -e_d\}$. By \Href{Lemma}{lem:Caratheodory_minus_point}, for each point $p \in \{\pm e_1, \dots, \pm e_d\},$ there are  $d$ points, say $v_{1}, \dots, v_{d},$ of $Q$ with the property 
$p \in  \conv{\braces{0,v_{1}, \dots, v_{d}}}.$ The convex hull of the union of such $d$-tuples of points for all $p \in \{\pm e_1, \dots, \pm e_d\},$  contains the cross-polytope and hence contains the ball $\frac{\ball{d}}{\sqrt{d}}.$  

Now, it suffices to apply \Href{Theorem}{thm:Steinitz_dropping_out_points} to go from $2d^2$ points to $2d$ points whose convex hull contains the ball
$\frac{1}{2(2d^2 +d) +1} \cdot \frac{\ball{d}}{\sqrt{d}} \supset \frac{d^{-\frac{5}{2}}}{7} \ball{d} .$ 
\end{proof}

\section*{Acknowledgments}
The author extends gratitude to M{\'a}rton Nasz{\'o}di for insightful discussions and comprehensive help. Special thanks are also due to J{\'a}nos Pach, who facilitated our collaborative efforts with M{\'a}rton.

\bibliographystyle{alpha}

\end{document}